\documentclass[reqno,centertags, 12pt]{amsart}
\usepackage{comment,graphicx}
\usepackage{color}

\definecolor{purple}{rgb}{0.65, 0, 0.9}
\definecolor{orange}{rgb}{1,.5,0}

\usepackage{graphics}
\usepackage{amssymb,amsmath,amsfonts}%numbysec}
-\marginparwidth \oddsidemargin -9mm \evensidemargin \oddsidemargin
\usepackage{latexsym}
\advance\hoffset by 5mm

\def\@abssec#1{\vspace{.05in}\footnotesize \parindent .2in
{\bf #1. }\ignorespaces}
\newtheorem{theorem}{Theorem}[section]

\newtheorem{lemma}[theorem]{Lemma}
\newtheorem{proposition}[theorem]{Proposition}
\newtheorem{corollary}[theorem]{Corollary}
\newtheorem{definition}[theorem]{Definition}
\newtheorem{remark}[theorem]{Remark}

\def \R {\mathbb R}
\def \T {\mathbb T}
\def \N {\mathbb N}

\def \Z {\mathbb Z}

\newcommand{\be}{\mathbf e}

\allowdisplaybreaks \numberwithin{equation}{section}

\renewcommand{\be}{\begin{equation}}
\newcommand{\ee}{\end{equation}}

\title[Small scale formations in the incompressible porous media equation]{Small scale formations in the incompressible porous media equation}
\author{Alexander Kiselev}
\thanks{Department of
Mathematics, Duke University, 120 Science Dr., Durham NC 27708, USA;
email: kiselev@math.duke.edu}
\author{Yao Yao}
\thanks{School of
Mathematics, Georgia Institute of Technology, Atlanta, GA 30332, USA; email: yaoyao@math.gatech.edu}

\begin{document}

%{\Large \bf PRELIMINARY DRAFT}

\begin{abstract}
We construct examples of solutions to the incompressible porous media (IPM) equation that must exhibit
infinite in time growth of derivatives provided they remain smooth.
As an application, this allows us to obtain nonlinear instability for a class of stratified steady states of IPM.\end{abstract}


\maketitle

\section{Introduction}\label{intro}

In this paper, we consider the 2D incompressible porous media (IPM) equation. The equation describes evolution of density carried by the flow of incompressible fluid
that is determined via Darcy's law in the field of gravity:
\begin{equation}\label{ipm}
\partial_t \rho +(u \cdot \nabla ) \rho =0, \,\,\, \nabla \cdot u =0, \,\,\, u +\nabla p = -(0,\rho).
\end{equation}
Here $\rho$ is the transported density, $u$ is the vector field describing the fluid motion, and $p$ is the pressure. Throughout this paper, we consider the spatial domain $\Omega$ to be one of the following: the whole space $\mathbb{R}^2$, the torus $\T^2 = [-\pi, \pi)^2$, or the bounded strip $S:=\T\times [-\pi,\pi]$ that is periodic in $x_1$. In the last case, due to the presence of boundaries, $u$ also satisfies $u\cdot n=0$ for $x_2=\pm \pi$. In all the three cases, one can obtain a more explicit Biot-Savart law for $u$: \[
u = \nabla^\perp (-\Delta_\Omega)^{-1} \partial_{x_1} \rho.
\]
Here $\nabla^\perp = (-\partial_{x_2}, \partial_{x_1})$, and the inverse Laplacian $(-\Delta_\Omega)^{-1}$ for $\Omega = \mathbb{T}^2$ and $\Omega = \T\times [-\pi,\pi]$ will be specified in Section~\ref{prelim}.

There have been many recent papers analyzing the well-posedness questions for the IPM equation and its variants \cite{CCL,CG,CGO,FGSV,YY}, lack of uniqueness of weak solutions \cite{CFG,LS},
and questions of long time dynamics \cite{ElgindiIPM, CCL}. Viewed as an active scalar, the IPM equation is less regular than the 2D Euler equation in vorticity form,
and has the same level of regularity as the SQG equation.
Local well-posedness for sufficiently regular initial data  has been proved in \cite{CGO}
for $\R^2$, and \cite{CCL} for the strip $S$. The argument can be adapted to the periodic setting $\T^2$; we will sketch a simple proof in Section~\ref{sec_torus}. The question of global regularity vs finite-time blow up
is open for the IPM equation, similarly to the SQG equation case. Moreover, to the best of our knowledge, there are not even examples of smooth solutions to the IPM
equation that have infinite growth of derivatives. There are plenty of such examples for the 2D Euler equation, going back to work of Yudovich \cite{Jud1,Yud2} (see e.g. \cite{Den1}, \cite{KS}
for more recent examples and further references). However, for the more singular SQG equation case, such examples have been established only recently \cite{HK}.
The reason for such delay is that an example of infinite in time creation of small scales requires sufficiently strong control of the solution, which is not easily achieved when
the drift is more singular. The example of \cite{HK} is based on the insight gained in the constructions for the 2D Euler case \cite{KS,Z1}, and is based on a hyperbolic point scenario
controlled by odd-odd symmetry of the active scalar. It is tempting to use a similar idea for the IPM equation, but its structure is different - in particular, odd symmetry in $x_2$ but even
symmetry of $\rho$ in $x_1$ is conserved instead of the odd-odd symmetry for the SQG equation. This, and the more detailed structure of the Biot-Savart law, appear to be significant obstacles
in extending ideas of \cite{HK,KS,Z1} to the IPM equation setting.
In this paper, we construct examples of infinite growth of derivatives in smooth solutions of the IPM equation using a different idea, exploiting existence of
monotone quantity which corresponds to the potential energy of the fluid.
All our estimates below assume that the solutions remain smooth; more specifically, the arguments work if $\rho$ and $u$ are at least Lipschitz. If this regularity fails in finite time, we already
have an even more dramatic effect than what we are trying to establish.

\subsection{Small scale formation in IPM}

In this paper, we consider the following three scenarios:

(S1) Let $\Omega = \mathbb{R}^2$. Assume $\rho_0\in C_c^\infty(\mathbb{R}^2)$ is odd in $x_2$, and $\rho_0\geq 0$ in $\mathbb{R}\times\mathbb{R}^+$.

(S2) Let $\Omega = \mathbb{T}^2=[-\pi,\pi)^2$ be the 2D torus. Assume $\rho_0 \in C^\infty(\mathbb{T}^2)$ is odd in $x_2$.

(S3) Let $\Omega = S:=\mathbb{T} \times [-\pi,\pi]$ be a bounded strip that is periodic in $x_1$. Assume $\rho_0 \in C^\infty(S)$.

Our first result shows that in the scenario (S1), $\rho(t)$ must have infinite-in-time growth in $\dot{H}^s$ norm for any $s>0$, if it remains regular for all times. Note that $s>0$ is the sharp threshold, since for $s=0$ we know $\|\rho(t)\|_{L^2}=\|\rho_0\|_{L^2}$ does not grow in time.

\begin{theorem}
\label{thm_whole_space}
For $\Omega=\mathbb{R}^2$, let $\rho_0 \not\equiv 0$ satisfy the scenario (S1). Assuming that there is a global-in-time smooth solution $\rho(x,t)$ to \eqref{ipm} with initial data $\rho_0$, we have
\begin{equation}
\label{whole_space_int}
\int_0^\infty \|\rho(t)\|_{\dot{H}^s(\R^2)}^{-\frac{4}{s}} dt \leq C(s, \rho_0)  \quad\text{ for all }s>0,
\end{equation}
which implies
\begin{equation}
\label{whole_space_ptwise}
\limsup_{t\to\infty} t^{-\frac{s}{4}} \|\rho(t)\|_{\dot{H}^s(\R^2)}=\infty \quad\text{ for all }s>0.
\end{equation}
\end{theorem}

The next result concerns the torus scenario (S2), where we prove infinite-in-time growth under some additional symmetry and positivity assumptions on $\rho_0$. As we will see in the remark afterwards, the same result also holds for the bounded strip scenario (S3).
\begin{theorem}
\label{thm_periodic}
For $\Omega=\mathbb{T}^2$, let $\rho_0\not\equiv 0$ satisfy the scenario (S2). In addition, assume that $\rho_0$ is even in $x_1$, $\rho_0=0$ for $x_1=0$, and $\rho_0\geq 0$ in $D:=[0,\pi]^2$. Assuming that there is a global-in-time smooth solution $\rho(x,t)$ to \eqref{ipm} with initial data $\rho_0$, we have
\begin{equation}
\label{periodic_int}
\int_0^\infty \|\partial_{x_1}\rho(t)\|_{\dot{H}^s(\T^2)}^{-\frac{2}{2s+1}} dt \leq C(s,\rho_0) \quad\text{ for all }s>-\frac{1}{2},
\end{equation}
which implies
\begin{equation}
\label{periodic_ptwise}
\limsup_{t\to\infty} t^{-(s+\frac{1}{2})} \|\rho(t)\|_{\dot{H}^{s+1}(\T^2)} \geq \limsup_{t\to\infty} t^{-(s+\frac{1}{2})} \|\partial_{x_1}\rho(t)\|_{\dot{H}^s(\T^2)}=\infty \quad\text{ for all }s>-\frac{1}{2}.
\end{equation}
\end{theorem}

\begin{remark}\label{rmk_strip}
Observe that the solution $\rho(\cdot,t)$ in $\T^2$ from Theorem~\ref{thm_periodic} is automatically a solution in the bounded strip $\T \times[-\pi,\pi]$, with $u$ satisfying the no-flow condition on the top and bottom boundaries.
(This is because $\rho(\cdot,t)$ is odd in $x_2$ and has period $2\pi$ in $x_2$ for all times. Thus $\rho(\cdot,t)$ is also odd about $x_2=\pm\pi$, implying $u(\cdot,t)\cdot e_2=0$ for all times on $x_2=\pm \pi$.)
Therefore, the growth results of Theorem~\ref{thm_periodic} directly hold in scenario (S3). We note that the local well-posedness for the scenario (S3) has been
established in \cite{CCL}, which in particular ensures the uniqueness of solution while it remains regular.
\end{remark}

\subsection{Nonlinear instability in IPM}

One can easily check that any horizontal stratified state $\rho_s=g(x_2)$ is a stationary solution of \eqref{ipm} in $\R^2$, $\T^2$ or $S$, since  $u=\nabla^\perp (-\Delta_\Omega)^{-1} \partial_{x_1} \rho \equiv 0$. (As we will see in Lemma~\ref{energy_lemma}, all smooth stationary solutions in $S$ are of the form $\rho_s=g(x_2)$. However, in $\R^2$ and $\T^2$ there are other smooth stationary solutions, e.g. any vertical stratified state is also stationary; see also \cite[Section 5]{CLV} for smooth stationary solutions in $\R^2$ supported in an infinite slanted strip.)

Below we briefly summarize the previous stability results for the horizontal stratified state $\rho_s = g(x_2)$.
Denoting $\eta(x,t):=\rho(x,t)-\rho_s(x)$ and plugging it into \eqref{ipm}, $\eta$ satisfies
\begin{equation}\label{IPM_perturb}
\partial_t \eta + u \cdot\nabla\eta = -g'(x_2) u_2,
\end{equation}
with $u=\nabla^\perp (-\Delta_\Omega)^{-1}\partial_{x_1}\eta$. For $\eta$ small, the linearized equation is $\partial_t \eta = -g'(x_2) u_2$, which can be written as
\[
\partial_t \eta = -g'(x_2) (-\Delta_\Omega)^{-1}\partial^2_{x_1}\eta.
\]
Since $(-\Delta_\Omega)^{-1}\partial^2_{x_1}$ is a negative operator, one would expect the equation to be linearly stable if $g'$ is uniformly negative (i.e. lighter density on top, heavier on the bottom).

For the stratified state $\rho_s=-x_2$, the asymptotic stability of the nonlinear equation \eqref{IPM_perturb} has been rigorously established by Elgindi \cite{ElgindiIPM} in $\R^2$ and Castro--C\'ordoba--Lear \cite{CCL} in $S$,
which also implies the global well-posedness of \eqref{ipm} for initial data close to $\rho_s=-x_2$ in certain Sobolev spaces. More precisely, for $\rho_s=-x_2$ in $\R^2$, if $\|\eta\|_{W^{4,1}(\R^2)}+\|\eta\|_{H^s(\R^2)}<\epsilon\ll 1$  for $s\geq 20$, \cite[Theorem 1.3]{ElgindiIPM} proved that $\eta$ remains regular for all time and satisfies $\|\eta(t)\|_{H^3}\lesssim \epsilon t^{-1/4}$ for all $t> 0$.  \cite{ElgindiIPM} also obtained asymptotic stability results for periodic perturbation $\eta \in H^s(\T^2)$, where $\rho_s$ is still $-x_2$ in the whole plane. In \cite{CCL}, the authors proved that for the stratified state $\rho_s=-x_2$ in $S$ is asymptotically stable in $H^s(S)$ for $s\geq 10$, although it may converge to a slightly different stratified state from $\rho_s$ as $t\to\infty$.

In this paper, we aim to prove two nonlinear instability results for the horizontal stratified steady state $\rho_s=g(x_2)$ in $\T^2$ and $S$ respectively.
What sets our approach apart is that we are not following the common path of converting linear instability into a nonlinear one. Rather, we use the monotone quantity - potential energy - to prove infinite-in-time growth of Sobolev norms and then leverage these results to conclude the nonlinear instability.
Our first instability result shows that in $\T^2$, \emph{any} horizontal stratified steady state $\rho_s$ that is odd in $x_2$ is nonlinearly unstable, and the instability can grow ``infinitely in time''.  Namely, for any arbitrarily large $k>0$, one can construct an initial data $\rho_0$ that is arbitrarily close to $\rho_s$ in $H^k$, such that $\limsup_{t\to\infty}\|\rho(t)-\rho_s\|_{\dot{H}^s(\T^2)}=\infty$ for all $s>1$.

\begin{theorem}\label{thm_instability_torus}
Let $\rho_s\in C^\infty(\T^2)$ be any horizontal stratified state (i.e. $\rho_s(x) = g(x_2)$) that is odd in $x_2$. For any $\epsilon>0$ and any $k>0$, there exists an initial data $\rho_0 \in C^\infty(\T^2)$ satisfying
\[
\|\rho_0 - \rho_s\|_{H^k(\T^2)} \leq \epsilon,
\]
such that the solution $\rho(\cdot,t)$ to \eqref{ipm} with initial data $\rho_0$ (provided it remains smooth for all times) satisfies
\begin{equation}\label{instab_torus_eq}
\limsup_{t\to\infty} t^{-\frac{s}{2}} \|\rho(t)-\rho_s\|_{\dot{H}^{s+1}(\T^2)} =\infty \quad\text{ for all }s> 0.
\end{equation}

\end{theorem}

Finally, we prove an instability result in the bounded strip $S=\T\times[-\pi,\pi]$ for \emph{any} stratified steady state $\rho_s\in C^\infty(S)$, including those monotone stratified states that are linearly stable such as $\rho_s = -x_2$. Namely,
we can construct a smooth perturbation small in $H^{2-\gamma}$ norm for any $\gamma>0$, such that $\limsup_{t\to\infty}\|\rho(t)-\rho_s\|_{\dot{H}^{s}(S)}=\infty$ for any $s>1$.

\begin{theorem}\label{thm_instability_strip}
Let $\rho_s\in C^\infty(S)$ be any stationary solution. For any $\epsilon,\gamma>0$, there exists an initial data $\rho_0 \in C^\infty(S)$ satisfying
\begin{equation}\label{eq_difference}
\|\rho_0 - \rho_s\|_{H^{2-\gamma}(S)} \leq \epsilon,
\end{equation}
such that the solution $\rho(\cdot,t)$ to \eqref{ipm} with initial data $\rho_0$ (provided it remains smooth for all times) satisfies
\begin{equation}\label{eq_instability_strip}
\limsup_{t\to\infty} t^{-\frac{s}{2}} \|\rho(t)-\rho_s\|_{\dot{H}^{s+1}(S)} =\infty \quad\text{ for all }s> 0.
\end{equation}
\end{theorem}

\begin{remark}It is a natural question whether the perturbation can be made arbitrarily small in higher Sobolev spaces. While it is unclear to us whether $H^{2}$ is the sharp threshold, we know that the exponent cannot exceed 10: for $\rho_s=-x_2$, if the initial perturbation is small in $H^{10}$ or above,  \cite{CCL} showed $\|\rho(t)-\rho_s\|_{H^3(S)}$ remains uniformly bounded in time.
\end{remark}

\subsection{Organization of the paper}
In Section~\ref{prelim} we discuss some preliminaries and the local well-posedness results in the scenarios (S1)--(S3). In  Section  3  we show the monotonicity of the potential energy in the three scenarios, and use it to prove the infinite-in-time growth results in Theorems~\ref{thm_whole_space}--\ref{thm_periodic}. We take a brief detour in Section 4 to derive some infinite-in-time growth results for less restrictive initial data, which we call the ``bubble'' solution and the ``layered'' solution. This will enable us to obtain nonlinear instability results in Section~\ref{sec_instability} for initial data close to stratified steady states, where we prove Theorems~\ref{thm_instability_torus}--\ref{thm_instability_strip}.
\color{black}

\section{Preliminaries on problem setting and local well-posedness}\label{prelim}
In this section, we discuss some preliminaries such as the Sobolev spaces for the spatial domains $\mathbb{R}^2$, $\T^2$ and $S=\T\times[-\pi,\pi]$ respectively, as well as the local-wellposedness results for the IPM equation \eqref{ipm}. For the whole space and strip case, the local-wellposedness theory have already been established in \cite{CG} and \cite{CCL} respectively. For the torus case we are unable to locate a local-wellposedness result, so we give a short proof in Section~\ref{sec_torus}.

\subsection{Sobolev norms and local well-posedness in $\R^2$} For any $f\in L^2(\R^2)$, its Fourier transform is defined as usual as
\[
\hat f(\xi) := \frac{1}{2\pi} \int_{\R^2} e^{-i\xi\cdot x} f(x) dx \quad\text{ for }\xi \in \R^2,
\]
and the Plancherel theorem yields
$
\|\hat f\|_{L^2(\R^2)}^2 = \| f\|_{L^2(\R^2)}^2.
$
As usual, we define
\[
\|f\|_{\dot{H}^s(\R^2)}^2 := \int_{\R^2} |\xi|^{2s} |\hat f(\xi)|^2 d\xi \quad\text{ for any }s\neq 0,
\]
and
\[\|f\|_{{H}^s(\R^2)}^2 :=  \int_{\R^2} (1+|\xi|^2)^s |\hat f(\xi)|^2 d\xi \quad\text{ for any }s\in \R.
\]

For \eqref{ipm} in $\R^2$, C\'ordoba--Gancedo--Orive \cite[Theorem 3.2]{CGO} proved local-wellposedness  for initial data $\rho_0\in H^s(\R^2)$ with $s>2$. They also established a regularity criteria, showing that $\rho(t)$ remains regular as long as $\int_0^t \|\nabla\rho(t)\|_{BMO\,} ds < \infty$.
They also obtained another regularity criteria with a geometric flavor, and we refer the reader to \cite[Theorem 3.4]{CGO} for details.

\subsection{Sobolev norms and local well-posedness in $\mathbb{T}^2$}
\label{sec_torus}

For any $f\in L^1(\T^2)$, let us denote its Fourier series as
\begin{equation}\label{fourier_periodic_1}
f(x) = \sum_{k\in\Z^2} \hat f(k) e^{ik\cdot x},
\end{equation}
where the Fourier coefficient $\hat f(k)$ for $k=(k_1,k_2)\in \Z^2$ is given by
\begin{equation}\label{fourier_periodic_2}
\hat f(k)=\frac{1}{(2\pi)^2} \int_{\T^2} e^{-ik\cdot x} f(x) dx.
\end{equation}
By Parseval's theorem, for any $f,g\in L^2(\T^2)$ we have $\int_{\T^2} f(x)\overline{g(x)} dx = (2\pi)^2  \sum_{k\in \Z^2} \hat f(k)\overline{\hat g(k)},$ which in particular implies
\begin{equation}\label{fourier_periodic_3}
\|f\|_{L^2(\T^2)}^2 = (2\pi)^2 \sum_{k\in \Z^2} |\hat f(k)|^2.
\end{equation}
For $s\neq 0$, throughout this paper, $\|f\|_{\dot{H}^s(\T^2)}$ is defined by
\begin{equation}\label{fourier_periodic_4}
\|f\|_{\dot{H}^s(\T^2)}^2 = (2\pi)^2 \sum_{k\in \Z^2\setminus\{(0,0)\}} |k|^{2s}|\hat f(k)|^2.
\end{equation}
%For $s \in \N $ and multi-indices $m=(m_1,m_2)$ with $m_1+m_2=s$ (and let $D^m := \partial_{x_1}^{m_1} \partial_{x_2}^{m_2}$), we have
%\[
%c(s) \|f\|_{\dot{H}^s(\T^2)}^2 \leq \sum_{\substack{m_1,m_2\in \N \cup\{0\} \\ m_1+m_2=s}} \| D^m f\|_{L^2(\T)}^2  \leq C(s)\|f\|_{\dot{H}^s(\T^2)}^2.
%\]
Finally, for mean-zero $f$ (in particular this is the case for $\partial_{x_1}\rho$), its inverse Laplacian is given by
$(-\Delta)^{-1} f = \sum_{k\in\Z^2\setminus\{(0,0)\}}  |k|^{-2} \hat f(k) e^{ik\cdot x}$.

Below we sketch a-priori estimates that can be used to establish local regularity as well as conditional criteria
for global regularity. With these estimates, a fully rigorous argument can be given in a standard way, using either smooth mollifier
approximations like in \cite{MB} or Galerkin approximations.
%We do not aim to get sharp conditions on the initial data or regularity criteria.

Suppose that $\rho(x,t)$ is a smooth solution of $\partial_t \rho + (u \cdot \nabla) \rho=0,$ $u = \nabla^\perp (-\Delta)^{-1} \partial_{x_1} \rho,$ where $\nabla^\perp = (-\partial_{x_2}, \partial_{x_1}).$
Observe that all the $L^p$ norms of $\rho$ are conserved by evolution. %, as well as more generally all $L^p$ norms.
Multiplying the equation by $(-\Delta)^s \rho$ and integrating we obtain
\[ \frac12 \partial_t \|\rho\|^2_{\dot{H}^s} + \int_{\T^2} (u \cdot \nabla) \rho (-\Delta)^s \rho\, dx=0. \]
Here $s \geq 1$ is an integer. In the integral above, we can expand the power of the Laplacian and integrate by parts, then use the periodicity to transfer exactly half of the derivatives on $u\cdot\nabla \rho$. What we get is a sum of terms of the form
\[ \int_{\T^2} D^s ((u \cdot \nabla) \rho) D^s \rho\,dx, \]
where $D^s$ stands for some partial derivative of order $s.$ Next we apply Leibniz rule to open up the derivative $D^s$ falling on $(u \cdot \nabla) \rho.$
Note that when all derivatives fall on $\rho,$ we get
\[ \int_{\T^2} (u \cdot \nabla) D^s \rho D^s \rho \,dx = \frac12 \int_{\T^2} (u \cdot \nabla) (D^s \rho)^2 \,dx =0 \]
due to incompressibility of $u.$ Therefore we obtain a sum of the terms of the form
\[ \int_{\T^2} D^j u D^{s-j+1} \rho D^s \rho \,dx \]
where $1 \leq j \leq s.$ Let us apply H\"older inequality to control such integral by
\[ \left| \int_{\T^2} D^j u D^{s-j+1} \rho D^s \rho \,dx \right| \leq \|D^j u \|_{L^{p_j}}\|D^{s-j+1} \rho\|_{L^{q_j}}\|D^s \rho\|_{L^2}, \]
where $p_j,$ $q_j$ satisfy $p_j^{-1}+q_j^{-1} = \frac12.$ Let us recall a particular case of Gagliardo-Nirenberg inequality \cite{Gag,Nir}
\begin{equation}\label{gn}
 \|D^j f\|_{L^q} \leq C \|f\|_{L^\infty}^{1-\theta} \|f\|_{\dot{H}^s}^\theta
\end{equation}
for any $f \in C_0^\infty(\R^2),$ where in the 2D case $\theta = \frac{j-\frac{2}{q}}{s-1}.$
The inequality is valid for $s>j,$ $0 <\theta < 1.$ While \eqref{gn} is usually stated in $\R^2,$ an extension to $\T^2$ is straightforward.
Taking now $p_j = \frac{2(s+1)}{j}$ and $q_j = \frac{2(s+1)}{s+1-j}$ and applying \eqref{gn}, we get
\[ \|D^j u \|_{L^{p_j}} \leq C\|D^j \rho\|_{L^{p_j}} \leq C \|\rho\|_{L^\infty}^{1-\frac{js}{s^2-1}} \|\rho\|_{\dot{H}^s}^{\frac{js}{s^2-1}}, \]
where in the first step we used $L^p-L^p$ bound on singular integral operators for $1<p<\infty.$ Similarly,
\[ \|D^{s+1-j} \rho \|_{L^{q_j}} \leq   C \|\rho\|_{L^\infty}^{1-\frac{(s+1-j)s}{s^2-1}} \|\rho\|_{\dot{H}^s}^{\frac{(s+1-j)s}{s^2-1}}. \]
Therefore, for all $1 \leq j \leq s$ we have
\[ \left| \int_{\T^2} D^j u D^{s-j+1} \rho D^s \rho \,dx \right| \leq C \|\rho\|_{L^\infty}^{\frac{s-2}{s-1}} \|\rho\|_{\dot{H}^s}^{\frac{2s-1}{s-1}}, \]
and hence
\begin{equation}\label{locHs}
 \partial_t \|\rho\|^2_{\dot{H}^s} \leq C \|\rho\|_{L^\infty}^{\frac{s-2}{s-1}} \|\rho\|_{\dot{H}^s}^{\frac{2s-1}{s-1}}.
\end{equation}
Such inequality can be used to show local well-posedness in $H^s$ provided that $s>2.$

To obtain a criteria for blow up, we can run a similar calculation but using $\|\nabla \rho\|_{L^\infty}$ and $\|\nabla u \|_{L^\infty}$ instead of $\|\rho\|_{L^\infty}$
and with $p_j= 2s/j,$ $q_j= 2s/(s-j).$ This way instead of \eqref{locHs} we obtain the differential inequality
\[  \partial_t \|\rho\|^2_{\dot{H}^s} \leq C \left( \|\nabla \rho\|_{L^\infty} +\|\nabla u\|_{L^\infty} \right) \|\rho\|_{\dot{H}^s}^2. \]
We can conclude control of $\|\rho\|_{\dot{H}^s}$ up to any time $T$ provided that $\int_0^T \left( \|\nabla \rho\|_{L^\infty} +\|\nabla u\|_{L^\infty} \right)\,dt$ remains finite.
Let us state a proposition summarizing the observations of this section.
\begin{proposition}\label{lwgr}
Consider the IPM equation \eqref{ipm} with the initial data $\rho_0 \in H^s(\T^2),$ $s>2$ an integer. Then there exists a time $T=T(\|\rho_0\|_{H^s})$ such that for all $0\leq t \leq T$
there exists a unique solution $\rho(x,t),u(x,t) \in C([0,T],H^s(\T^2)).$ Moreover, the solution blows up at time $T$ if and only if
\[  \int_0^t \left( \|\nabla \rho(\cdot, r)\|_{L^\infty(\T^2)} +\|\nabla u(\cdot, r)\|_{L^\infty(\T^2)} \right)\,dr \rightarrow \infty \]
when $t \rightarrow T.$
\end{proposition}
\bf Remarks. \rm 1. Uniqueness of the solution can be shown in a standard way; blow up is understood in the sense of leaving the class $C([0,T],H^s(\T^2)).$  \\
2. The Proposition can certainly be improved in terms of the condition on $s$ and the regularity criterion, but we do not pursue it in this paper.

%\color{purple}

\subsection{Sobolev norms and local well-posedness in a strip}
When the domain is a bounded strip $S:=\T \times [-\pi,\pi]$, due to the presence of the top and bottom boundaries, the functional spaces and the local-wellposedness results are  more involved than the periodic case. Below we briefly describe the results by Castro--C\'ordoba--Lear \cite{CCL}, and we refer the readers to the paper for more details.

\emph{Biot-Savart law and functional space.}
In the strip case, the velocity field $u$ is given by $u = \nabla^\perp \psi$, where the stream function $\psi$ solves the Poisson's equation with zero boundary condition (see \cite[Section 2.2]{CCL} for a derivation):
\begin{equation}\label{poisson}
\begin{cases}
-\Delta\psi(\cdot,t) = \partial_{x_1} \rho(\cdot,t) &\text{ in }S\times[0,T),\\
\psi(\cdot,t) = 0 & \text{ on }\partial S\times[0,T),
\end{cases}
\end{equation}
so that $u=\nabla^\perp(-\Delta_S)^{-1}\partial_{x_1}\rho$.
One can check that the operator $-\Delta_{S}$ (with zero boundary condition) is a positive self-adjoint operator, and it has a family of eigenfunctions $\{\omega_{p,q}\}_{p\in\Z, q\in\N}$ that form an orthonormal basis for $L^2(S)$, given by
\[
\omega_{p,q}(x):= a_p(x_1) b_q(x_2) \quad\text{ for } p\in \Z, q \in \N,
\]
where
\[
a_p(x_1) := e^{ipx_1} \quad\text{ with }x_1\in \T \text{ and }p\in \Z,
\]
and
\[
b_q(x_2):= \begin{cases} \cos(\frac{qx_2}{2}) & q \text{ odd,}\\ \sin(\frac{qx_2}{2}) & q \text{ even,}\end{cases} \text{ with } x_2\in[-\pi,\pi] \text{ for }q\in \N.
\]
 The eigenfunction expansion allows us to define $(-\Delta_S)^{s} f$ for any $s\in \R$ and $f\in L^2(S)$. We can then define the $\dot{H}^s$ homogenous Sobolev norm as
\begin{equation}\label{def_sobolev_strip}
\|f\|_{\dot{H}^s(S)}^2 := \int_{S} f (-\Delta_S)^{s}f dx \quad\text{ for } s\in \R,
\end{equation}
and one can check that the spaces $\dot{H}^s(S)$ and  $\dot{H}^{-s}(S)$ are dual
with respect to the $L^2$ norm
(see e.g. \cite{Simon} for the general construction of a scale of Sobolev spaces associated with a positive self-adjoint operator).

For $s\geq 0$, let us define $\|f\|_{H^s(S)}^2 := \|f\|_{\dot{H}^s(S)}^2+\|f\|_{L^2(S)}^2$ as usual. For $s \in \N $, the above definition of
$ \|f\|_{{H}^s(S)}^2$ is comparable to $\sum_{0\leq m\leq s} \| D^m f\|_{L^2(S)}^2$
if $\partial_{x_2}^n f|_{\partial S}=0$  for all even $n$ with $n<s$, where $D^m$ is any partial derivative of order $m\leq s$.

\emph{Local/global well-posedness results in the strip.}
Since the goal of \cite{CCL} was to establish stability results near the steady state $\rho_s(x) = -x_2$,   \eqref{ipm} was written into an equivalent equation \eqref{IPM_perturb} (with $g(x_2)=-x_2$) describing the evolution of $\eta = \rho-\rho_s$.
Here $u$ can be expressed in terms of $\eta$ similarly to \eqref{poisson}, except that the right hand side $\partial_{x_1} \rho$ is replaced by $\partial_{x_1}\eta$ (using that $\rho_s$ has zero contribution to $u$ since it is a steady state).

When the initial data $\eta(0)$ belongs to the functional space $X^k(S)$, given by
\[
X^k(S) := \{f\in H^k(\T^2): \partial_{x_2}^n f|_{\partial S}=0 \text{ for all even $n$ with $n<k$}\},
\]
(which in fact coincides with $H^k(S)$ defined above),
the authors proved local-wellposedness of \eqref{IPM_perturb} for $\eta(0) \in X^k(S)$ for any $k\geq 3$
\cite[Theorem 4.1]{CCL}, and gave a regularity criteria showing that $\eta(t)$ remains in $X^k(S)$ as long as $\int_0^t (\|\nabla\eta(s)\|_{L^\infty(S)} + \|\nabla u(s)\|_{L^\infty(S)}) ds < \infty$. As we discussed in the introduction, for $k\geq 10$, they proved the asymptotic stability of $\eta(t)$ (which implies global regularity) for $\eta(0) \in X^{k}(S)$ with $\|\eta(0)\|_{X^k(S)}\ll 1$.

\color{black}

\section{Infinite-in-time growth in the IPM}\label{growth}

In this section we aim to prove Theorems \ref{thm_whole_space} and \ref{thm_periodic}. Throughout this paper, the evolution of the potential energy
\[ E(t) := \int_{\Omega} x_2 \rho(x,t)\,dx. \]
plays a key role. Let us first prove a simple lemma showing that in each of the scenarios (S1)--(S3), $E(t)$ is monotone decreasing in time, and its time derivative is integrable in $t\in(0,\infty)$.

\begin{lemma}\label{energy_lemma}
Assume that $\Omega$ and $\rho_0$ satisfy one of the scenarios (S1)--(S3). Assuming that there is a global-in-time smooth solution $\rho(x,t)$ to \eqref{ipm} with initial data $\rho_0$, we have
\[
\frac{d}{dt} E(t) = -\underbrace{\|\partial_{x_1} \rho\|_{\dot{H}^{-1}(\Omega)}^2}_{=:\delta(t)}.
\]
In addition, we have $\int_0^\infty \delta(t) dt\leq C(\rho_0) < \infty$.
\end{lemma}
\begin{proof}
A direct computation gives
\begin{equation} \label{eq_dEdt}
E'(t) = \frac{d}{dt} \int_{\Omega} x_2 \rho(x,t) \,dx = - \int_{\Omega} x_2 (u \cdot \nabla)\rho \,dx =  \int_{\Omega} u_2 \rho \,dx.
\end{equation}
Here the last inequality is due to the divergence theorem, where the boundary integral vanishes in all the three scenarios (S1)--(S3): In (S1), the boundary integral (at infinity) vanishes since $\rho(\cdot,t) \in C_c^\infty(\mathbb{R}^2)$ for all time. In (S2), the boundary integral    $ \int_{-\pi}^\pi x_2 u_1 \rho \,dx_2 \big|_{x_1=-\pi}^{x_1 =\pi}=0$ due to periodicity, and $\int_{-\pi}^{\pi} x_2 u_2 \rho \,dx_1 \big|_{x_2=-\pi}^{x_2 =\pi} =0$ since $\rho\equiv 0$ on $x_2=\pm \pi$ (which follows from the facts that $\rho(\cdot,t)$ is odd in $x_2$, and periodic in $\T^2$). In (S3), again $ \int_{-1}^1 x_2 u_1 \rho \,dx_2 \big|_{x_1=-\pi}^{x_1 =\pi}=0$ due to periodicity in $x_1$, whereas $\int_{-\pi}^{\pi} x_2 u_2 \rho \,dx_1 \big|_{x_2=-1}^{x_2 =1} =0$ due to $u\cdot n=0$ on $x_2=\pm \pi$.

By \eqref{eq_dEdt} and the Biot-Savart law $u=\nabla^\perp (-\Delta_{\Omega})^{-1}\partial_{x_1}\rho$, in (S1)--(S3), we get
\begin{equation}\label{Ecomp2}
E'(t) = \int_{\Omega} u_2 \rho \,dx = \int_{\Omega} \rho \partial_{x_1}(-\Delta_{\Omega})^{-1} \partial_{x_1}\rho \,dx = -\|\partial_{x_1} \rho\|_{\dot{H}^{-1}(\Omega)}^2,
\end{equation}
thus $E(t)$ is monotone decreasing. Note that in the strip case $\Omega=S$, the last identity follows from the definition of the $\dot{H}^{-1}$ norm in $S$ as in \eqref{def_sobolev_strip}.

\begin{comment}
\textcolor{orange}{Need to show that $E'(t) = -\|\partial_{x_1} \rho\|_{\dot{H}^{-1}}^2$ for $\Omega = \T \times[-\pi,\pi]$. In this case we still have $E'(t)  = \int_{\Omega} \rho \partial_{x_1}(-\Delta_{\Omega})^{-1}\partial_{x_1}\rho \,dx$, but need to define more carefully in Sec 2.3 what is the $H^{-1}$ norm there.}
\end{comment}

Moreover, $E(t)$ is uniformly bounded below for all times. In (S1), the assumptions that $\rho_0$ is odd in $x_2$ and  $\rho_0\geq 0$ in $\R\times \R^+$ yield that $\rho(t)\geq 0$ in $\R\times \R^+$, thus $E(t)\geq 0$ for all times. In (S2) and (S3), since a smooth solution $\rho(x,t)$ of \eqref{ipm} has its $L^\infty$ norm invariant in time, we have that
\[ E(t) \geq -\|\rho_0\|_{L^\infty(\Omega)} 4\pi \int_{0}^{\pi} x_2 dx_2 = -2\pi^3 \|\rho_0\|_{L^\infty(\Omega)} \quad\text{ for all }t\geq 0. \]
Hence in all three cases (S1)--(S3), we have
\begin{equation}\label{intH}
\int_0^\infty\delta(t) \,dt = \int_0^\infty  \|\partial_{x_1} \rho(\cdot, t)\|_{\dot{H}^{-1}(\Omega)}^2 \,dt = E(0)-\lim_{t\to\infty}E(t) \leq C(\rho_0)\quad\text{ for all }t\geq 0,
\end{equation}
finishing the proof.
\end{proof}
\begin{remark}
When the equation is set in $\R^2$ with $\rho$ decaying sufficiently fast, monotonicity of $E(t)$ has been derived in \cite[Corollary 1.2]{ElgindiIPM}. For the Muskat equation (which can be seen as a ``patch'' solution of IPM) with surface tension, \cite{JKM} uses the gradient flow structure to construct weak solutions, where the energy functional is the potential energy plus the surface area of the free boundary.
\end{remark}
\color{black}

Now we are ready to prove Theorems~\ref{thm_whole_space} and \ref{thm_periodic}.
\begin{proof}[\textbf{\textup{Proof of Theorem~\ref{thm_whole_space}}}]
Due to Lemma~\ref{energy_lemma}, $\delta(t) := \|\partial_{x_1} \rho(t)\|_{\dot{H}^{-1}(\R^2)}^2$ satisfies $\int_0^\infty \delta(t) dt < C_0(\rho_0)  < \infty$.
Denoting $C_2 := \|\rho_0\|_{L^2(\R^2)}^2$, we have $\|\rho(t)\|_{L^2(\R^2)}^2=C_2$ for all $t\geq 0$, since the $L^p$ norm of $\rho$ is invariant in time for all $1\leq p\leq \infty$. Let us define
$
I := \left\{t\in (0,\infty): \delta(t) < \frac{1}{4}C_2 \right\}.$ Note that $\int_0^\infty \delta(t) dt \leq C_0(\rho_0)$ directly implies $|\mathbb{R}^+\setminus I| \leq 4C_0(\rho_0)/C_2$. We claim that for any $s>0$,
\begin{equation}\label{eq_goal}
\|\rho(t)\|_{\dot H^s(\R^2)} \geq C(s,\|\rho_0\|_{L^1(\R^2)}, C_2) \delta(t)^{-\frac{s}{4}} \quad\text{ for any }t\in I.
\end{equation}
Once we prove \eqref{eq_goal}, plugging it into $\int_0^\infty \delta(t) dt \leq C_0(\rho_0)$ and using the fact that $|\mathbb{R}^+\setminus I| \leq 4C_0(\rho_0)/C_2$, we have
\[
\int_0^\infty \|\rho(t)\|_{\dot{H}^s(\R^2)}^{-\frac{4}{s}} dt \leq C(s,\|\rho_0\|_{L^1}, C_2)^{-\frac{4}{s}} \int_I \delta(t) dt + \frac{4C_0(\rho_0)}{C_2} \left(\inf_{t\geq 0}\|\rho(t)\|_{\dot{H}^s(\R^2)}\right)^{-\frac{4}{s}} \leq C(s, \rho_0).
\]
Here in the last inequality we used that for any $s>0$, $\|\rho(t)\|_{\dot{H}^s(\R^2)}$ is bounded below by a positive constant $c(s,\rho_0)$, which follows from the elementary interpolation inequality $\|\rho(t)\|_{L^2(\R^2)} \leq \|\rho(t)\|_{\dot H^s(\R^2)}^{\frac{1}{1+s}} \|\rho\|_{L^1(\R^2)}^{\frac{s}{1+s}}$, as well as the fact that $\|\rho(t)\|_{L^2(\R^2)}$ and $\|\rho(t)\|_{L^1(\R^2)}$ are invariant in time. This finishes the proof of \eqref{whole_space_int}. Combining \eqref{whole_space_int} with the fact that  $\int_1^\infty t^{-1}dt=\infty$ gives \eqref{whole_space_ptwise} as a direct consequence.

In the rest we aim to prove \eqref{eq_goal} for any fixed $t\in I$, and we will drop the $t$ dependence in $\rho$ and $\delta$ below for notational simplicity.
 Defining
\[
D_\delta:= \left\{(\xi_1,\xi_2): \frac{|\xi_1|}{|\xi|} \geq \sqrt{\frac{2\delta}{C_2}}\right\}.
\]
we observe that
\[
 \delta = \|\partial_{x_1} \rho\|_{\dot{H}^{-1}(\R^2)}^2 = \int_{\mathbb{R}^2} \frac{\xi_1^2}{|\xi|^2} |\hat\rho(\xi)|^2 d\xi \geq \int_{D_\delta}  \frac{\xi_1^2}{|\xi|^2} |\hat\rho(\xi)|^2 d\xi \geq \frac{2\delta}{C_2} \int_{D_\delta} |\hat \rho|^2 d\xi.
\]
This gives $\int_{D_\delta} |\hat \rho|^2 d\xi \leq \frac{1}{2}C_2$, and combining it with $\|\hat \rho\|_{L^2(\R^2)}^2= \| \rho\|_{L^2(\R^2)}^2 = C_2 $ yields $\int_{D_\delta^c} |\hat \rho|^2 d\xi \geq \frac{1}{2}C_2$. Note that $D_\delta^c$ consists of two symmetric cones containing the $\xi_2$ axis, and it can be expressed in polar coordinates as $D_\delta^c = \{(r \cos\theta,r\sin\theta): r\geq 0, |\cos\theta| < \sqrt{2\delta/C_2}$\}.

Clearly, $\|\hat\rho\|_{L^\infty(\R^2)} \leq (2\pi)^{-1}\|\rho_0\|_{L^1(\R^2)} =: C_1$. Let $h_\delta>0$ be such that $|D_\delta^c \cap \{|\xi_2|<h_\delta\}| = (4C_1^2)^{-1} C_2$, which will be estimated momentarily. Such definition gives
\[
\int_{D_\delta^c \cap \{|\xi_2|\geq h_\delta\}} |\hat\rho|^2 d\xi = \int_{D_\delta^c}  |\hat\rho|^2 d\xi - \int_{D_\delta^c \cap \{|\xi_2|<h_\delta\}}  |\hat\rho|^2  d\xi \geq \frac{1}{2}C_2  - (4C_1^2)^{-1} C_2 C_1^2 = \frac{1}{4}C_2,
\]
immediately leading to
\begin{equation}\label{Hs}
\|\rho\|_{\dot{H}^s(\R^2)}^2 \geq \int_{\mathbb{R}^2} |\xi_2 |^{2s} |\hat\rho|^2 d\xi \geq h_\delta^{2s} \int_{D_\delta^c \cap \{|\xi_2|\geq h_\delta\}} |\hat\rho|^2 d\xi \geq \frac{C_2}{4} h_\delta^{2s}.
\end{equation}
It remains to estimate $h_\delta$. Denoting $\theta_0 := \cos^{-1}(\sqrt{\frac{2\delta}{C_2}})$, we know $D_\delta^c \cap \{|\xi_2|<h_\delta\}$ consists of two identical triangles with height $h_\delta$ and base $2h_\delta \cot\theta_0$. Thus
\[
 (4C_1^2)^{-1} C_2 = |D_\delta^c \cap \{|\xi_2|<h_\delta\}| = 2h_\delta^2 \cot\theta_0 \leq 4\sqrt{\delta} C_2^{-1/2}h_\delta^2,
\]
where in the last inequality we used  $\cos\theta_0 = \sqrt{\frac{2\delta}{C_2}}$ and $\sin\theta_0 = \sqrt{1-\frac{2\delta}{C_2}} \geq 1/\sqrt{2}$, due to $t\in I$.
Therefore $h_\delta \geq (4C_1)^{-1} C_2^{3/4} \delta^{-1/4}$. Plugging it into \eqref{Hs} yields \eqref{eq_goal}, finishing the proof.
\end{proof}

\begin{proof}[\textbf{\textup{Proof of Theorem~\ref{thm_periodic}}}]
Since $\rho_0$ is even in $x_1$ and odd in $x_2$, due to the Biot-Savart law $u = \nabla^\perp(-\Delta)^{-1} \partial_{x_1}\rho$,  the even-odd symmetry of $\rho$ remains true for all times. In particular, it implies that  on the boundary of the smaller square $D:= [0,\pi]^2$, we have $u(\cdot,t)\cdot n|_{\partial D}=0$, and combining it with $\rho_0\geq 0$ on $D$ gives $\rho(t)\geq 0$ on $D$ for all times. In addition, note that $\rho_0=0$ on $x_1=0$ and the fact that $u(\cdot,t)\cdot n|_{\partial D}=0$  imply $\rho(0,x_2,t)\equiv 0$ for all $x_2$ and $t$.

For any $t\geq 0$, $\rho(\cdot,t):\mathbb{T}^2\to\mathbb{R}$ has Fourier series \eqref{fourier_periodic_1}--\eqref{fourier_periodic_2} (with $f$ replaced by $\rho(\cdot,t)$), and the Fourier coefficient $\hat\rho(k,t)$ for $k=(k_1,k_2)\in \Z^2$ can be written as
\begin{equation}\label{rho_g}
\begin{split}
\hat\rho(k,t)
 &= \frac{1}{(2\pi)^2} \int_{\T} e^{-ik_1 x_1} \int_{\T} e^{-ik_2 x_2} \rho(x_1, x_2, t) dx_2 dx_1\\
 &= \frac{1}{(2\pi)^2}  \int_{\T} e^{-ik_1 x_1} (-2i) \underbrace{\int_{0}^\pi \sin(k_2 x_2) \rho(x_1, x_2, t) dx_2}_{=:g(x_1,k_2,t)} dx_1,
\end{split}
\end{equation}
where the last identity follows from the oddness of $\rho(\cdot, t)$ in $x_2$.

Let us take a closer look at the function $g(x_1,1,t)$ in the last line of \eqref{rho_g} (where we set $k_2=1$). It satisfies the following properties for all $t\geq 0$:
\begin{enumerate}
\item[(a)] $g(x_1,1,t)\geq 0$ for all $x_1\in\mathbb{T}$ and $t\geq 0$, and is even in $x_1$.
\item[(b)] $g(0,1,t)=0$ for all $t\geq 0$.
\item[(c)] $\int_{\T} g(x_1, 1, t) dx_1 \geq c(\rho_0)$ for all $t\geq 0$, where $c(\rho_0)>0$ only depends on $\rho_0$.
\end{enumerate}
Here properties (a, b) follow from the facts that $\rho(t)$ is even in $x_1$, nonnegative on $D= [0,\pi]^2$, and $\rho(0,x_2,t)\equiv 0$ for all times. For property (c), note that
$
\int_{\T} g(x_1,1,t) dx_1 = 2 \int_0^\pi g(x_1,1,t) dx_1 = 2\int_D \sin(x_2) \rho(x,t) dx.
$
H\"older's inequality and the fact that $\sin(x_2)\rho(x,t)\geq 0$ in $D$ yield that
\[
\int_D \sin(x_2) \rho(x,t) dx \geq  \left(\int_D \sin(x_2)^{-1/2} dx\right)^{-2} \left(\int_D \rho(x,t)^{1/3} dx\right)^3 \geq c \left(\int_D \rho(x,t)^{1/3} dx\right)^3,
\]
where $c>0$ is a universal constant. Since $\rho$ is advected by a divergence-free flow $u$ with $u\cdot n|_{\partial D}=0$, one can easily check that $\int_D \rho(x,t)^{1/3} dx = \int_D \rho_0^{1/3} dx>c_1$ for some $c_1(\rho_0)>0$, finishing the proof of property (c).

Let us  define $\hat g(k_1,t)$ as the Fourier coefficient of $g(\cdot, 1,t)$, given by
\begin{equation}\label{rho_g_2}
\hat g(k_1,t) :=\frac{1}{2\pi} \int_{\T} e^{-ik_1 x_1} g(x_1,1,t) dx \quad\text{ for any }k_1\in\mathbb{Z}.
\end{equation}
Comparing \eqref{rho_g_2} with \eqref{rho_g} directly yields
\begin{equation}
\label{eq_k_rho}
 \hat\rho(k_1,1,t)=\frac{-2i}{2\pi}\hat g(k_1,t)\quad\text{ for any }k_1\in \Z.
 \end{equation} Using the functions $g$ and $\hat g$, we can estimate $\delta(t) = \|\rho(t)\|_{\dot{H}^{-1}(\T^2)}^2$ from below as
\begin{equation}
\begin{split}
\delta(t) & = (2\pi)^2 \sum_{k\in \Z^2\setminus\{(0,0)\}}\frac{k_1^2}{|k|^2}|\hat\rho(k_1,k_2,t)|^2  \geq (2\pi)^2\sum_{k_1\in \Z\setminus\{0\}} \frac{k_1^2}{k_1^2+1}|\hat\rho(k_1,1,t)|^2\\
&\geq 2\sum_{k_1\in \Z\setminus\{0\}} |\hat g(k_1,t)|^2 = \frac{1}{\pi} \int_{\T} |g(x_1,1,t)-\bar g(t)|^2 dx_1
\end{split}
\end{equation}
where
$
\bar g(t) := \frac{1}{2\pi}\int_{\T} g(x_1,1,t) dx_1
$ is the average of $g(\cdot,1,t)$ in $\T$.
By property (c), we have $\bar g(t)\geq \frac{c(\rho_0)}{2\pi}>0$ for all $t\geq 0$. Intuitively, if $\delta(t)$ is small, $g(\cdot,1,t)$ must be very close to $\bar g(t)$ in $L^2$. With $g(0,1,t)$ pinned down at zero (by property (b)), and $\bar g$ being uniformly positive, $g(\cdot,1,t)$ must have order 1 oscillations in a small neighborhood near $0$, suggesting it should have a large $\dot{H}^\alpha$ norm for $\alpha>\frac{1}{2}$. This estimate will be made rigorous in Lemma~\ref{lem:1d} right after the proof. Applying Lemma~\ref{lem:1d} to $g(x_1,1,t)$, we have
\begin{equation}\label{ineq_g}
\|g(\cdot,1,t)\|_{\dot{H}^{\alpha}(\T)} \geq c(\alpha, \rho_0) \delta(t)^{-\alpha+\frac{1}{2}} \quad\text{ for all }\alpha>\frac{1}{2}.
\end{equation}
Note that
\[
\|g(\cdot,1,t)\|_{\dot{H}^\alpha(\T)}^2 = 2\pi \sum_{k_1\neq 0} |k_1|^{2\alpha}|\hat g(k_1,t)|^2 = 2\pi^3 \sum_{k_1\neq 0} |k_1|^{2\alpha}|\hat\rho(k_1,1,t)|^2 \leq\frac{\pi}{\sqrt{2}}  \|\partial_{x_1}\rho\|_{\dot{H}^{\alpha-1}(\T^2)}^2,
\]
where the last inequality follows by just looking at the $k_2=1$ part of the sum for the last norm taken on Fourier side and using $\alpha>1/2.$
Setting $s:=\alpha-1$ and applying \eqref{ineq_g}, we have
\[
\|\partial_{x_1}\rho\|_{\dot{H}^{s}} \geq \Big(\frac{\sqrt{2}}{\pi}\Big)^{1/2} \|g(\cdot,1,t)\|_{\dot{H}^{s+1}(\T)} \geq c(s,\rho_0) \delta^{-s-\frac{1}{2}} \quad\text{ for } s> -\frac{1}{2}.
\]
Plugging this inequality into $\int_0^\infty\delta(t)dt\leq C_0(\rho_0)<\infty$ implies \eqref{periodic_int}, and combining \eqref{periodic_int} with the fact that  $\int_1^\infty t^{-1}dt=\infty$ gives \eqref{periodic_ptwise} as a direct consequence.
\end{proof}

Now we state and prove the lemma used in the proof of Theorem~\ref{thm_periodic}.
\begin{lemma}\label{lem:1d}
If $f:\mathbb{T}\to\mathbb{R}$ satisfies $f(0)=0$, $\int_{\T} f(x) dx \geq c_0>0$ and $\int_{\T} |f-\bar f|^2 dx < \delta$ (where $\bar f := \frac{1}{2\pi}\int_{\T} f(x)dx$), then
\begin{equation}\label{eq_lemma1}
\|f\|_{\dot{H}^\alpha(\T)} \geq c(\alpha, c_0) \delta^{-\alpha+\frac{1}{2}} \quad\text{ for all } \alpha>\frac{1}{2}.
\end{equation}
\end{lemma}
\begin{proof}
Note that $h(x):=f(x)-\bar f$ has mean zero in $\T$, and $h(0) = -\bar f \leq -c_0$. By the assumption $\int_{\T} h^2 dx < \delta$, there exists some $x_0 \in (0,\frac{4\delta}{c_0^2})$ such that $h(x_0)> -\frac{c_0}{2}$. This implies
\[
\|h\|_{C^\gamma(\T)} \geq \frac{|h(x_0)-h(0)|}{|x_0-0|^\gamma} \geq (c_0/2)^{1+2\gamma}\delta^{-\gamma} \quad\text{ for all }0<\gamma\leq 1.
\]
Applying the Sobolev embedding theorem, we have
\[
\|h\|_{\dot H^{\gamma+\frac{1}{2}}(\T)} \geq c(\gamma) \|h\|_{C^\gamma(\T)} \geq c(\gamma, c_0) \delta^{-\gamma} \quad\text{ for all }0<\gamma\leq 1,
\]
and setting $\alpha=\gamma+\frac{1}{2}$ gives \eqref{eq_lemma1} for $\alpha \in (\frac{1}{2}, \frac{3}{2}]$, where we also use that $\|h\|_{\dot H^{\alpha}(\T)}=\|f\|_{\dot H^{\alpha}(\T)}$. For $\alpha>\frac{3}{2}$, we can apply the interpolation inequality $\|h\|_{\dot{H}^1(\T)} \leq \|h\|_{\dot{H}^\alpha(\T)}^{\frac{1}{\alpha}} \|h\|_{L^2(\T)}^{\frac{\alpha-1}{\alpha}}$ to obtain
\[
\|h\|_{\dot{H}^\alpha(\T)} \geq \|h\|_{\dot{H}^1(\T)}^{\alpha} \|h\|_{L^2(\T)}^{-(\alpha-1)} \geq c(\alpha,c_0) \delta^{-\alpha+\frac{1}{2}},
\]
thus we can conclude.
\end{proof}

\section{Bubble and layer solutions}
Previously, we have proved infinite time growth results in Theorem~\ref{thm_periodic} and Remark~\ref{rmk_strip} for scenarios (S2) and (S3), under some additional assumption on $\rho_0$. In this section we aim to work with less restrictive initial data -- in particular, the assumption that $\rho_0=0$ for $x_1=0$ can now be dropped. This will enable us to obtain instability results in Section~\ref{sec_instability} for initial data close to stratified steady states. However, as the proof is done by a different approach, the set of Sobolev exponents with norm growth (as well as the growth rate in time) is not as good as Theorem~\ref{thm_periodic} and Remark~\ref{rmk_strip}.

We first consider the initial data $\rho_0\in C^\infty(\T^2)$ of ``bubble'' type, that is, its level sets  have a connected component  $\Gamma_0$ enclosing a simply-connected region, and $|\nabla\rho_0|$ does not vanish on $\Gamma_0$ (see Figure~\ref{fig_bubble} for an illustration). Intuitively, since the topological structure of all level sets is preserved under the evolution, the presence of the ``bubble" prevents the solution $\rho(x,t)$ from aligning into a perfectly stratified form where $\partial_{x_1} \rho$ may increasingly vanish. In the next result we rigorously justify this by showing that $\|\partial_{x_1}\rho(t)\|_{L^1(\T^2)} > c >0$ for all times, and as we will see, this leads to infinite-in-time growth in certain Sobolev norms.

\begin{proposition}\label{prop_bubble}
Let $\Omega=\mathbb{T}^2$, and assume $\rho_0$ satisfies the scenario (S2). Suppose there exists a simple closed curve $\Gamma_0 \subset \T^2$ enclosing a simply-connected domain $D_0\subset \T^2$, and $\rho_0$ satisfies $\rho_0|_{\Gamma_0}=\text{const}$ and $\inf_{\Gamma_0} |\nabla\rho_0|>0$\footnote{Observe that by Sard's theorem \cite{Sard}, since $\rho_0 \in C^2(\T^2)$, the set of $h$ such that $\{\rho_0(x)=h\}$ contains a critical point has Lebesgue measure zero.
}. Assuming that there is a global-in-time smooth solution $\rho(x,t)$ to \eqref{ipm} with initial data $\rho_0$, we have
\begin{equation}
\label{periodic_int_bubbles}
\int_0^\infty \|\partial_{x_1} \rho(\cdot, t)\|_{\dot{H}^s(\T^2)}^{-\frac{2}{s}} dt \leq  C(s,\rho_0)\quad\text{ for all }s>0.
\end{equation}
which implies
\begin{equation}
\label{periodic_ptwise_bubbles}
 \limsup_{t\to\infty} t^{-\frac{s}{2}} \|\partial_{x_1}\rho(t)\|_{\dot{H}^s(\T^2)}=\infty \quad\text{ for all }s>0.
\end{equation}
\end{proposition}

\begin{proof}
Since $\rho_0\in C^\infty(\T^2)$ with $\inf_{\Gamma_0} |\nabla\rho_0|>0$,  $|\nabla \rho_0|$ is uniformly positive in some open neighborhood of $\Gamma_0$. Combining this with $\rho_0|_{\Gamma_0}=\text{const}=:c_0$, for any $c\in\R$ sufficiently close to $c_0$, the level set $\{\rho_0=c\}$ has a connected component that is a simple closed curve near $\Gamma_0$. Since $\Gamma_0$ encloses a simply-connected region $D_0$, there exists a simple closed curve $\Gamma_1 \subset D_0$ such that $\rho_0|_{\Gamma_1}=c_1\neq c_0$. Denote by $D_1$ the region enclosed by $\Gamma_1$, which is also simply-connected. See Figure~\ref{fig_bubble} for an illustration of the curves $\Gamma_0, \Gamma_1$ and the domains $D_0, D_1$.

\begin{figure}
\begin{center}
\includegraphics[scale=1.2]{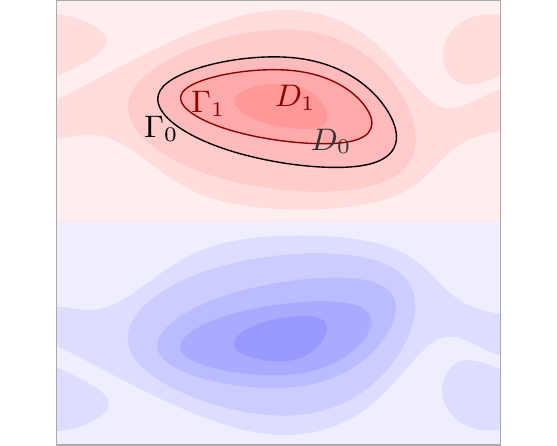}
\caption{An illustration of the curves $\Gamma_0$, $\Gamma_1$ and the domains $D_0$, $D_1$ in the ``bubble solution''.\label{fig_bubble}}
\end{center}
\end{figure}

Let us as usual define the trajectories of the flow by
\begin{equation}\label{trajectory}
 \frac{d\Phi_t(x)}{dt} = u(\Phi_t(x),t), \,\,\,\Phi_0(x)=x.
\end{equation}
While the solution remains smooth, the flow map $\Phi_t:\T^2\to\T^2$ is a measure-preserving smooth mapping.
Thus $\Phi_t (D_0)$ and $\Phi_t (D_1)$ remain simply-connected in $\T^2$, and they satisfy $\Phi_t (D_1) \subset \Phi_t (D_0)$ for all $t\geq 0$.  Denoting by $m(\cdot)$ the  Lebesgue measure of a set (which is preserved by $\Phi_t$), we have $m(\Phi_t (D_1))=m(D_1)$ for all $t\geq 0$. In addition, since $\rho$ is advected by $u$ \eqref{ipm}, we have as usual that $\rho(\Phi_t(x),t) = \rho_0(x)$ for all $x$ and $t$, thus $\rho(x,t)|_{x\in \Phi_t(\Gamma_0)}=c_0$ and $\rho(x,t)|_{x\in \Phi_t(\Gamma_1)}=c_1$ for all $t\geq 0$.

Let us denote $\Pi_2:\T^2\to \T$  the projection map onto the $x_2$ variable, i.e. for any $S\subset \T^2$, $\Pi_2(S) := \{x_2\in \T: (x_1,x_2)\in S \text{ for some }x_1\in\T\}$. Using $\Phi_t (D_1) \subset \Phi_t (D_0)$, we have
\[
\Pi_2(\Phi_t (D_1)) \subset \Pi_2(\Phi_t (D_0)) \quad\text{ for all }t\geq 0.
\]
Since $\Phi_t (D_0)$ and $\Phi_t (D_1)$ are simply-connected domains enclosed by boundaries
 $\Phi_t (\Gamma_0)$ and $\Phi_t (\Gamma_1)$ respectively, the above becomes
 \begin{equation}\label{eq_sets}
 \Pi_2(\Phi_t (\Gamma_1)) = \Pi_2(\Phi_t (D_1)) \subset \Pi_2(\Phi_t (D_0)) = \Pi_2(\Phi_t (\Gamma_0)) \quad\text{ for all }t\geq 0.
\end{equation}
Using
$m(\Phi_t (D_1))=m(D_1)$, we have $\Pi_2(\Phi_t (D_1)) \geq \frac{m(D_1)}{2\pi}$ for all $t\geq 0$.
Finally, defining $I(t) :=  \Pi_2(\Phi_t (\Gamma_1))$, which is a subset in $\T$, we have shown that
$|I(t)|\geq \frac{m(D_1)}{2\pi}$ and $I(t) \subset \Pi_2(\Phi_t (\Gamma_0))$ for all $t\geq 0$.

By definition of $I(t)$ and \eqref{eq_sets}, for any $t\geq 0$ and $x_2\in I(t)$, $\T\times x_2$ has a non-empty intersection with both $\Phi_t(\Gamma_1)$ and $\Phi_t(\Gamma_0)$. Since $\rho(\cdot,t)|_{\Phi_t(\Gamma_0)}=c_0$ and $\rho(\cdot,t)|_{\Phi_t(\Gamma_1)}=c_1$, it implies
\begin{equation}\label{eq_temp_int}
\int_{\T} |\partial_{x_1} \rho(x_1,x_2,t)| dx_1 \geq |c_1-c_0| \quad\text{ for any } x_2\in I(t), t\geq 0.
\end{equation}
Integrating this in $x_2$ and using $|I(t)|\geq \frac{m(D_1)}{2\pi}$, we have
$ \int_{\T^2} |\partial_{x_1} \rho(x,t)| \,dx \geq \frac{m(D_1) |c_1-c_0|}{2\pi}$  for all $t\geq 0, $
thus Cauchy-Schwartz yields
\begin{equation}\label{czrhox1}
\int_{\T^2} |\partial_{x_1} \rho|^2 \, dx \geq \frac{1}{4\pi^2} \left(\int_{\T^2} |\partial_{x_1} \rho| \,dx\right)^2 = \frac{m(D_1)^2 |c_1-c_0|^2}{16\pi^4} >0 \quad\text{ for all }t\geq 0.
\end{equation}
Applying the interpolation inequality%\begin{equation}\label{Hn}
$ \|f\|_{L^2(\T^2)} \leq \|f\|_{\dot{H}^{-1}(\T^2)}^{\frac{s}{s+1}} \|f\|_{\dot{H}^s(\T^2)}^{\frac{1}{s+1}} $
%\end{equation}
for $s>0$ with $f=\partial_{x_1}\rho$, we have
\begin{equation}\label{Hn}
\delta(t) = \|\partial_{x_1}\rho\|_{\dot{H}^{-1}(\T^2)}^2 \geq \|\partial_{x_1}\rho\|_{L^2(\T^2)}^{2+\frac{2}{s}}\|\partial_{x_1}\rho\|_{\dot{H}^s(\T^2)}^{-\frac{2}{s}} \quad\text{ for all }s>0, t\geq 0.
\end{equation}
Plugging \eqref{Hn}, \eqref{czrhox1} into \eqref{intH} we obtain \eqref{periodic_int_bubbles}. Finally, combining \eqref{periodic_int_bubbles} with the fact that  $\int_1^\infty t^{-1}dt=\infty$ gives \eqref{periodic_ptwise_bubbles} as a direct consequence. \end{proof}

The growth for ``bubble'' solutions can be easily adapted to the bounded strip case as follows.

\begin{corollary}\label{cor_bubble}
Let $\Omega=S=:\T\times[-\pi,\pi]$, and assume $\rho_0$ satisfies scenario (S3). Suppose there exists a simple closed curve $\Gamma_0 \subset S^\circ$ enclosing a simply-connected domain $D_0\subset S$, and $\rho_0$ satisfies $\rho_0|_{\Gamma_0}=\text{const}$ and $\inf_{\Gamma_0} |\nabla\rho_0|>0$. Assuming that there is a global-in-time smooth solution $\rho(x,t)$ to \eqref{ipm} with initial data $\rho_0$, we have
\begin{equation}
\label{strip_int_bubbles}
\int_0^\infty \|\partial_{x_1} \rho(\cdot, t)\|_{\dot{H}^s(S)}^{-\frac{2}{s}} dt \leq  C(s,\rho_0)\quad\text{ for all }s>0.
\end{equation}
which implies
\begin{equation}
\label{strip_ptwise_bubbles}
 \limsup_{t\to\infty} t^{-\frac{s}{2}} \|\partial_{x_1}\rho(t)\|_{\dot{H}^s(S)}=\infty \quad\text{ for all }s>0.
\end{equation}
\end{corollary}

\begin{proof}
The proof is almost identical to the proof of Proposition~\ref{prop_bubble}. Again, there exists $\Gamma_1$ enclosing a simply-connected domain $D_1 \subset D_0$, such that $\rho_0|_{\Gamma_1}= c_1$, and $c_1\neq c_0:=\rho_0|_{\Gamma_0}$. Defining $I(t) :=  \Pi_2(\Phi_t (\Gamma_1))$ (which is now a subset in $[-\pi,\pi]$), the same argument gives that $|I(t)|\geq \frac{m(D_1)}{2\pi}$ and \eqref{eq_temp_int}. Thus \eqref{czrhox1} still holds (except that the integral now takes place in $S$ instead of $\T^2$), and the rest of the proof remains unchanged.
Note that the interpolation inequality $ \|f\|_{L^2(S)} \leq \|f\|_{\dot{H}^{-1}(S)}^{\frac{s}{s+1}} \|f\|_{\dot{H}^s(S)}^{\frac{1}{s+1}}$ holds in $S$ as well by a straightforward argument using the eigenfunction expansion
similar to the standard Fourier argument in $\T^2.$
\end{proof}

\color{black}

Our next result concerns ``layered'' initial data, which we define below.

\begin{definition}\label{def_layer}
For $\rho_0\in C^\infty(\T^2)$, we say it has a \emph{layered structure} if there exists a measure-preserving smooth diffeomorphism $\phi:\T^2\to\T^2$ that satisfies $\phi(\T\times\{\pi\})=\T\times\{\pi\}$, such that $\rho_s = \rho_0(\phi^{-1}(x))$ is a \emph{stratified solution}, i.e. $\rho_s(x)$ only depends on $x_2$. In this case, we call $\rho_s$ the \emph{stratified state corresponding to $\rho_0$}.
\end{definition}
Note that any layered initial data $\rho_0$ has a unique corresponding stratified state $\rho_s$. (Even though the mapping $\phi$ is not unique, e.g. one can shift $\phi$ by any $(a,0)$). To see this, take any curve $\Gamma$ such that $\rho_0|_{\Gamma}=c$ with $\inf_{\Gamma} |\nabla\rho_0| > 0$, and denote by $D$ the region bounded between $\Gamma$ and $\T\times\{\pi\}$. Since $\phi$ is measure-preserving and $\phi(\T\times\{\pi\})=\T\times\{\pi\}$, we know $\phi(D)=\T\times[x_2,\pi]$ must have the same area as $D$, leading to $\rho_s(\pi-\frac{|D|}{2\pi})=c$.
Since $\rho_0 \in C^\infty(\T^2)$ (thus $\rho_s$ is also smooth), Sard's theorem allows us to run this argument for a.e. $c$, which defines $\rho_s$ uniquely for all $x$.
See Figure~\ref{fig_layer} for an illustration of a layered initial data $\rho_0$ and its corresponding stratified state $\rho_s$.

\begin{figure}[htbp]
\begin{center}
\includegraphics[scale=1]{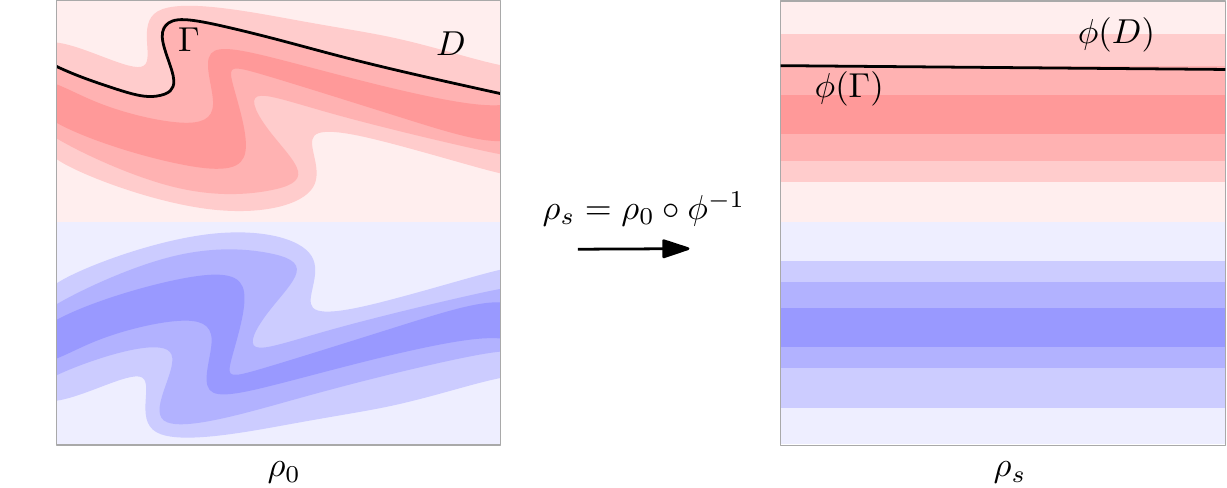}
\caption{An illustration of a ``layered'' initial data $\rho_0$ and its corresponding stratified state $\rho_s$.\label{fig_layer}}
\end{center}
\end{figure}

Clearly, Theorem~\ref{thm_periodic} cannot be applied to any layered $\rho_0$ since $\rho_0(0,x_2)\not\equiv 0$, and  Proposition~\ref{prop_bubble} fails too since there is no level set $\Gamma$ enclosing a simply-connected region. Despite these difficulties, we will show that small scale formation can still happen to $\rho_0$, as long as its potential energy is strictly lower than that of $\rho_s$.

\begin{proposition}\label{prop_layer}
Let $\Omega=\mathbb{T}^2$. Assume $\rho_0$ satisfies scenario (S2), and it has a layered structure in the sense of Definition~\ref{def_layer}, with corresponding stratified state denoted by $\rho_s$. In addition, suppose
\begin{equation}\label{layer_energy}
E(0) = \int_{\T^2} \rho_0(x) x_2 dx < \int_{\T^2} \rho_s(x) x_2 dx =: E_s.
\end{equation}
 Then the estimates \eqref{periodic_int_bubbles} and \eqref{periodic_ptwise_bubbles} hold, given there is a global-in-time smooth solution $\rho(x,t)$ to \eqref{ipm} with initial data $\rho_0$.
\end{proposition}

\begin{remark}
 As a side note, it is not hard to construct layered initial data $\rho_0$ satisfying \eqref{layer_energy} -- see Figure~\ref{fig_layer_2} for an illustration.
\end{remark}

\begin{figure}[htbp]
\begin{center}
\includegraphics[scale=1]{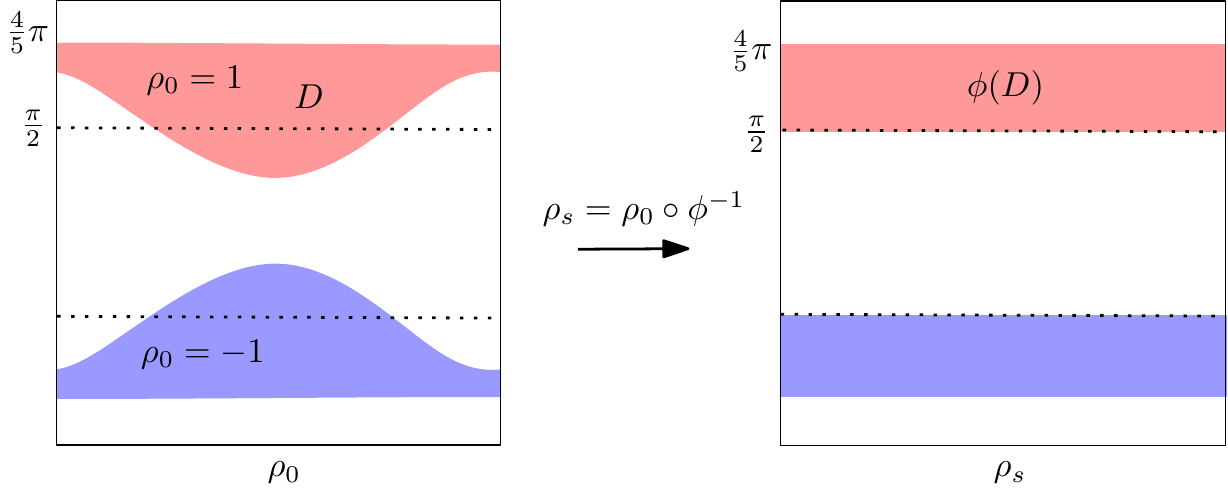}
\caption{An example of a ``layered'' $\rho_0$ satisfying the inequality \eqref{layer_energy}. Let $\rho_0$ be odd in $x_2$, and set $\rho_0=1_D$ in the upper half of $\T^2$, where $D$ is bounded between $x_2=\frac{4}{5}\pi$ and $x_2 = \frac{\pi}{2}-\frac{\pi}{4}\cos( x_1)$. The right figure shows its corresponding stratified state $\rho_s$. Such (discontinuous) $\rho_0$ has a potential energy strictly less than that of $\rho_s$, since the center of mass of $D$ is lower than $\phi(D)$. One can then slightly mollify $\rho_0$ to obtain a smooth layered initial data satisfying \eqref{layer_energy}. \label{fig_layer_2}}
\end{center}
\end{figure}

\begin{proof} To begin with, we will show that for any $t\geq 0$, $\rho(\cdot,t)$ also has a layered structure with corresponding stratified state being $\rho_s$. The assumption on $\rho_0$ gives $\rho_s\circ\phi = \rho_0$ for some measure-preserving diffeomorphism $\phi$. Combining this with $\rho(\Phi_t(x),t)=\rho_0(x)$ (where $\Phi_t(x)$ is the flow trajectory given by \eqref{trajectory}), we have $\rho_s = \rho((\Phi_t \circ \phi^{-1})(x),t)$ for all times. Here $\Phi_t \circ \phi^{-1}:\T^2\to\T^2$ is a measure-preserving diffeomorphism, and it keeps the set $\T\times\{\pi\}$  invariant, since both $\phi$ and $\Phi_t$ have this property:  $\phi$ has this property due to Definition~\ref{def_layer}, whereas $\Phi_t$ has this property since $u_2=0$ on $x_2=\pi$ for all times.

Let us denote
$
b := E_s - E(0) >0,
$ where the strict positivity is due to \eqref{layer_energy}. By Lemma~\ref{energy_lemma},  $E(t)$ is non-increasing in time, thus
\begin{equation}\label{eq_b}
E_s - E(t) \geq E_s - E(0) = b > 0 \quad\text{ for all }t\geq 0.
\end{equation}
Since $\rho_s$ is the only stratified state that is topologically reachable from $\rho(\cdot,t)$ (among all measure-preserving diffeomorphisms that keeps $\T\times\{\pi\}$ invariant), the fact that $\rho(\cdot, t)$ has a potential energy strictly less than $\rho_s$ (with the gap being at least $b$) intuitively suggests that $\rho(\cdot,t)$ cannot have all level sets very close to horizontal. Below we will show that this is indeed true, in the sense that $\int_{\T^2} |\partial_{x_1}\rho(x,t)|dx$ is bounded below by a positive constant for all times.

By \eqref{eq_b} and the definition of potential energy $E(t)$, we have
\begin{equation}\label{eq_temp00}
b \leq E_s - E(t) = \int_{\T^2} x_2 (\rho_s(x) -\rho(x,t)) dx \leq \pi \int_{\T^2}
|\rho_s(x) -\rho(x,t) |dx,
\end{equation}
and let us take a closer look at the integrand. In the first paragraph of the proof we showed $\rho_s(x)=\rho(\Psi_t(x),t)$, where $\Psi_t := \Phi_t \circ \phi^{-1}$ is a measure-preserving diffeomorphism that keeps $\T\times\{\pi\}$ invariant. As a result, for any $t\geq 0$ and $x_2\in \T$, $\Psi_t(\T\times\{x_2\})$ must have a non-empty intersection with $\T\times\{x_2\}$, i.e. there exists $\tilde x_1$ and $\bar x_1$ depending on $x_2$ and $t$, such that $\Psi_t(\tilde x_1,x_2) = (\bar x_1, x_2)$. Combining this with the fact that $\rho_s$ is a function of $x_2$ only, we have
\[
|\rho_s(x)-\rho(x,t)| = |\rho_s(\tilde x_1, x_2)-\rho(x,t)|  = |\rho(\bar x_1,x_2,t)-\rho(x_1,x_2,t)| \leq \int_{\T} |\partial_{x_1}\rho(s_1, x_2, t)| ds_1\]
for any $x=(x_1,x_2)\in \T^2, t\geq 0$.
Plugging this into \eqref{eq_temp00} gives
\[
b \leq 2\pi^2 \int_{\T^2} |\partial_{x_1}\rho(x,t)|dx,
\]
leading to $\int_{\T^2} |\partial_{x_1}\rho(x,t)|dx \geq \frac{b}{2\pi^2}>0$ for all $t\geq 0$.
Now that we have a positive lower bound on $\|\partial_{x_1} \rho(t)\|_{L^1(\T^2)}$, the rest of the argument can proceed the same way as in \eqref{czrhox1} and \eqref{Hn} in the proof of Proposition~\ref{prop_bubble}, allowing us to obtain the same estimates \eqref{periodic_int_bubbles}--\eqref{periodic_ptwise_bubbles}.
\end{proof}

\section{Instability of horizontally stratified steady states}

In this section we aim to prove the two nonlinear instability results Theorems~\ref{thm_instability_torus} and \ref{thm_instability_strip} in $\T^2$ and $S$ respectively. We start with Theorem~\ref{thm_instability_torus}, which shows that any stratified steady state $\rho_s$ in $\T^2$ that is odd in $x_2$ is nonlinearly unstable in an arbitrarily high Sobolev space. The idea is to locate a point $x_0 \in \T^2$ where locally $\rho_s$ has heavier density on top of lighter one, then use a circular flow to slightly perturb $\rho_s$ near $x_0$ to construct a ``layered'' initial data that satisfies the assumption of Proposition~\ref{prop_layer}.
\color{black}
\label{sec_instability}

\begin{proof}[\textbf{\textup{Proof of Theorem~\ref{thm_instability_torus}}}]
We claim that for any $\epsilon>0$ and $k>0$, we can construct a $\rho_0 \in C^\infty(\T^2)$ satisfying all the following:
\begin{enumerate}
\item[(a)] $\rho_0$ is odd in $x_2$, and has a layered structure in the sense of Definition~\ref{def_layer} with corresponding stratified state $\rho_s$.
\item[(b)] $\|\rho_0 - \rho_s\|_{H^k(\T^2)} < \epsilon$.
\smallskip
\item[(c)] $E(0) < \int_{\T^2} x_2\rho_s dx =: E_s$.
\end{enumerate}

Once these are shown to be true, a direct application of Proposition~\ref{prop_layer} immediately yields the infinite-in-time growth results \eqref{periodic_int_bubbles} and \eqref{periodic_ptwise_bubbles}. Since $\partial_{x_1}\rho(\cdot,t)= \partial_{x_1}(\rho(\cdot,t)-\rho_s)$, \eqref{periodic_ptwise_bubbles} directly implies \eqref{instab_torus_eq}, finishing the proof.

In the rest of the proof we aim to construct $\rho_0$ and prove the claim.
We will focus on the construction of $\rho_0$ in the upper half of torus $\T^2_+ := \T \times [0,\pi]$, and at the end we will extend it to the lower part $\T^2_-$ by an odd extension.

Recall that $\rho_s(x)=g(x_2)$ is a smooth stratified state that is odd in $x_2$.  Thus $g$ is odd and smooth in $\T$. Such $g$ cannot be monotone, so there exists some $h_0 \in (0,\pi)$ such that $g'(h_0) >0$.
For $0<\epsilon_0\ll 1$ to be fixed later, let $\varphi_{\epsilon_0} \in C^\infty_c(\R)$ be non-negative and supported on $[\epsilon_0, 2\epsilon_0]$. Let $v:\T^2_+ \to \R^2$ be the velocity field of an incompressible circular flow around $x_0 := (0, h_0)$, given by
\begin{equation}\label{def_v}
v(x) = (x-x_0)^\perp \varphi_{\epsilon_0}(|x-x_0|) \quad\text{ for }x\in\T^2_+.
\end{equation}
For any $\tau\geq 0$, let $\tilde\rho(\cdot, \tau)$ be the solution to
\begin{equation}\label{transport_v}
\partial_{\tau}\tilde\rho + v\cdot\nabla\tilde\rho = 0 \quad \text{ in }\T^2_+ \times(0,\infty)
\end{equation}
 with initial data $\tilde\rho(\cdot,0)=\rho_s$. Since $v$ is supported in a small annulus $B(x_0, 2\epsilon_0) \setminus B(x_0,\epsilon_0)$, clearly $\tilde\rho(\cdot,\tau)=\rho_s$ outside the annulus.
Intuitively, since $\rho_s$ has heavier density on top of lighter density locally near $x_0$, we formally expect that  $\tilde \rho(\tau)$ should have lower potential energy than $\rho_s$ for a short time. (Here the ``time'' $\tau$ is the perturbation parameter, and has nothing to do with the actual time in \eqref{ipm}).  Below we will rigorously show
\begin{equation}\label{energy_perturb}
F(\tau) := \int_{\T^2_+}  (\tilde\rho(x, \tau) -\rho_s) x_2 dx < 0 \quad\text{ for all }0<\epsilon_0\ll 1 \text{ and  }0<\tau\ll 1.
\end{equation}
Since the integral can be reduced to the set $B(x_0, 2\epsilon_0) \setminus B(x_0,\epsilon_0)$, it is convenient to write it in polar coordinates centered at the point $x_0$. Using the change of variables $x_1 = r\cos\theta, x_2 = h_0 + r\sin\theta$, we have
\begin{eqnarray}
\nonumber F(\tau) &=& \displaystyle\int_{\epsilon_0}^{2\epsilon_0} \int_0^{2\pi} \tilde\rho( r\cos\theta, h_0 + r\sin\theta, \tau) (h_0 + r\sin\theta) r d\theta dr - \underbrace{\int_{B(x_0, 2\epsilon) \setminus B(x_0,\epsilon)} \rho_s x_2 dx}_{=: C_s}\\
\nonumber &=& \displaystyle\int_{\epsilon_0}^{2\epsilon_0} \int_0^{2\pi} \rho_s(r\cos(\theta-\varphi_{\epsilon_0}(r)\tau), h_0 + r\sin(\theta-\varphi_{\epsilon_0}(r)\tau))  (h_0 + r\sin\theta) r d\theta dr - C_s\\
\nonumber &=& \displaystyle\int_{\epsilon_0}^{2\epsilon_0} \int_0^{2\pi} g(h_0 + r\sin(\theta-\varphi_{\epsilon_0}(r)\tau))  (h_0 + r\sin\theta) r d\theta dr - C_s\\
&=:& \displaystyle\int_{\epsilon_0}^{2\epsilon_0} f(r,\tau) dr - C_s, \label{eq_f}
\end{eqnarray}
where the second identity follows from the facts that $\tilde\rho(\cdot,\tau)$ is transported by $v$ with initial data $\rho_s$, and $v$ is a circular flow with angular velocity $\varphi_{\epsilon_0}(r)$ along $\partial B(x_0, r)$.
We can rewrite $f(r,\tau)$ as
\[
\begin{split}
f(r,\tau) = & r^2 \int_0^{2\pi} g(h_0 + r\sin(\theta-\varphi_{\epsilon_0}(r)\tau)) \sin\theta d\theta + h_0 r  \int_0^{2\pi} g(h_0 + r\sin(\theta-\varphi_{\epsilon_0}(r)\tau))  d\theta,
\end{split}
\]
where the second integral is constant in $\tau$ using the substitution $\vartheta = \theta-\varphi_{\epsilon_0}(r)\tau$. Thus taking the $\tau$ derivative gives
\[
\frac{d}{d\tau} f(r,\tau) = -r^3 \varphi_{\epsilon_0}(r) \int_0^{2\pi} g'(h_0 + r\sin(\theta-\varphi_{\epsilon_0}(r)\tau))  \cos(\theta-\varphi_{\epsilon_0}(r) \tau) \sin\theta \,d\theta,
\]
leading to
\begin{equation}\label{derivative_temp}
\frac{d}{d\tau} f(r,\tau)\Big|_{\tau=0} = -r^3 \varphi_{\epsilon_0}(r) \int_0^{2\pi} g'(h_0 + r\sin\theta)  \cos\theta \sin\theta \,  d\theta = 0,
\end{equation}
since the integrand is odd about $\theta=\frac{\pi}{2}$. Taking one more derivative and setting $\tau=0$, we have
\[
\begin{split}
\frac{d^2}{d\tau^2} f(r,\tau)\Big|_{\tau=0} &= r^4 \varphi_{\epsilon_0}^2(r)\int_0^{2\pi} g''(h_0 + r\sin\theta)  (\cos\theta)^2 \sin\theta\, d\theta \\
&\quad - r^3 \varphi_{\epsilon_0}^2(r)\int_0^{2\pi} g'(h_0 + r\sin\theta) (\sin \theta)^2 d\theta\\
&=\varphi_{\epsilon_0}^2(r)\left(  -\pi r^3  g'(h_0) + O(r^4)\right).
\end{split}
\]
Since $h_0$ is chosen such that $g'(h_0)>0$, for all sufficiently small $0<\epsilon_0\ll 1$ we have
\[
\frac{d^2}{d\tau^2} f(r,\tau)\Big|_{\tau=0} \leq -\frac{1}{2}\pi r^3 \varphi_{\epsilon_0}^2(r) g'(h_0)  < 0 \quad\text{ for all }r\in (\epsilon_0, 2\epsilon_0).
\]
Plugging \eqref{derivative_temp} and the above  into \eqref{eq_f} gives that $\frac{d}{d\tau} F(\tau)\big|_{\tau=0} =0$ and $\frac{d^2}{d\tau^2} F(\tau)\big|_{\tau=0} \leq -c(\epsilon_0) g'(h_0) < 0$. Combining these with $F(0)=0$ gives \eqref{energy_perturb}.

Finally, we use odd reflection to extend $\tilde\rho(\cdot,\tau)$ to $\T^2_-$, which is equivalent with simutaneously applying a circular flow to $\rho_s$ near $(0,-h_0)$ in the opposite direction as $v|_{\T^2_+}$. We then set $\rho_0 := \tilde\rho(\tau)$ with $0<\tau\ll 1$ sufficiently small, and let us check that it satisfies the claim (a,b,c): (a) is a direct consequence from the definition, since $\rho_0$ can be reached from $\rho_s$ by an explicit measure-preserving smooth flow that is only non-zero near $(0,\pm h_0)$. Also, since $ \tilde\rho(\cdot,\tau)$ is transported from $\rho_s$ with a smooth velocity field $v$, for any $k>0$, we have $\|\tilde\rho(\tau)-\rho_s\|_{H^k}\to 0$ as $\tau\to 0^+$, thus property (b) is satisfied.
As for the potential energy, note that \eqref{energy_perturb} gives that $E(0) - E_s = 2F(\tau) < 0$ when $0<\tau\ll 1$ is sufficiently small, finishing the proof of (c).
\end{proof}

Finally, we are ready to prove Theorem~\ref{thm_instability_strip} which deals with the instability on the strip. The idea is to perturb the steady state to make a small ``bubble'' localized near one point, then apply Corollary~\ref{cor_bubble}.

\begin{proof}[\textbf{\textup{Proof of Theorem~\ref{thm_instability_strip}}}]
First note that any stationary solution $\rho_s \in C^\infty(S)$ must be stratified of the form $\rho_s=g(x_2)$, since only in this case it satisfies $\|\partial_{x_1} \rho_s\|_{\dot{H}^{-1}(\Omega)}=0$ by Lemma~\ref{energy_lemma}.

Let $\varphi \in C_c^\infty(\R^2)$ be a nonnegative function  supported in $B(0,1)$ with $\varphi(0)=1$.
 For $0<\lambda< 1$, let
\[
\rho_{0\lambda}(x) := \rho_s(x) + 2A \lambda \varphi(\lambda^{-1} x) \quad\text{ for }x\in S,
\]
where $A := \|\nabla\rho_s\|_{L^\infty(S)}$.
Clearly, $\rho_{0\lambda}\in C^\infty(S)$, and $\rho_{0\lambda}=\rho_s$ in $S\setminus B(0,\lambda)$.

Let us first check that \eqref{eq_difference} is satisfied for $\rho_0:=\rho_{0\lambda}$ with $0<\lambda\ll 1$. A simple scaling argument yields that $\|D^2  (\rho_{0\lambda}-\rho_s)\|_{L^2(S)} = 2A\|D^2\varphi\|_{L^2(\R^2)}$ is invariant in $\lambda$ (where $D^2$ is any partial derivative of order 2), thus $\|\rho_{0\lambda}-\rho_s\|_{H^2(S)}$ is uniformly bounded for all $0<\lambda<1$. Combining this with $\|\rho_{0\lambda}-\rho_s\|_{L^2(S)}\leq CA \lambda^2$, we have $\|\rho_{0\lambda}-\rho_s\|_{H^{2-\gamma}(S)} \leq CA\lambda^\gamma$ for all $\gamma>0$ (where $C$ only depends on $\varphi$), where the right hand side can be made arbitrarily small for $0<\lambda\ll 1$.

We claim that for any $0<\lambda< 1$, $\rho_{0\lambda}$ satisfies the assumption of a ``bubble solution'' in Corollary~\ref{cor_bubble}. To see this, note that the definitions of $A$ and $\rho_{0\lambda}$ yields
\[
 \rho_{0\lambda}(x) = \rho_s(x) \leq \rho_s(0)+ A \lambda \quad\text{ for any }x\in\partial B(0,\lambda),
\]
whereas
\[
\rho_{0\lambda}(0) = \rho_s(0)+2A\lambda.
\]
Applying Sard's theorem \cite{Sard} to $\rho_{0\lambda}$, for almost every  $h\in (\rho_s(0)+ A \lambda,  \rho_s(0)+2A\lambda)$, we know $\{\rho_{0\lambda}=h\}$ has a connected component in $B(0,\lambda)$ on which $|\nabla \rho_{0\lambda}|$ never vanishes. Naming any such connected component $\Gamma_0$, we then have that $\rho_{0\lambda}$ satisfies the assumption in Corollary~\ref{cor_bubble}. As a result we have
the estimate \eqref{strip_ptwise_bubbles}.
Using that $\partial_{x_1}\rho(\cdot,t)=\partial_{x_1}(\rho(\cdot,t)-\rho_s)$, \eqref{strip_ptwise_bubbles} directly implies \eqref{eq_instability_strip}, thus finishes the proof.
\end{proof}

\begin{remark}
For a stratified solution $\rho_s=g(x_2)$ that does not satisfy $g'\leq 0$, the perturbation can be made small in higher Sobolev spaces. Namely, if there exists $x_0\in S$ such that $\partial_{x_2}\rho_s(x_0)>0$, one can proceed as in the proof of Proposition~\ref{prop_layer} to construct a ``layered'' initial data close to $\rho_s$ in $H^k$ norm for arbitrarily large $k>0$.
\end{remark}

\color{black}

\noindent {\bf Acknowledgement.} \rm The authors acknowledge partial support of the NSF-DMS grants 1715418, 1846745 and 2006372. AK has been partially supported by Simons Foundation. YY has been partially supported by the Sloan Research Fellowship.
This paper has been initiated at the AIM Square, and the authors thank AIM for support and collaborative opportunity.

\end{document}